\documentclass[12pt]{amsart}

\usepackage{amssymb}
\usepackage{amsmath}
\usepackage{booktabs}
\usepackage{youngtab}

\usepackage{amsthm}
\newtheorem{thm}{Theorem}
\newtheorem{cor}{Corollary}
\newtheorem{lem}{Lemma}

\newcommand{\B}{{\mathcal B}}
\newcommand{\T}{{\mathcal T}}
\newcommand{\F}{{\mathbb F}}
\newcommand{\da}{{\text{-}}}
\newcommand{\qbinom}[3]{\genfrac{[}{]}{0pt}{}{#1}{#2}_{#3}}
\newcommand{\subspaces}[2]{\genfrac{[}{]}{0pt}{}{#1}{#2}}
\DeclareMathOperator\GL{GL}
\DeclareMathOperator\borel{B}

\begin{document}

\title{$q$-Analogs of $t$-Wise Balanced Designs from Borel Subgroups}
\author{Michael Braun}
\address{University of Applied Sciences Darmstadt, Germany, Faculty of Computer Science, michael.braun@h-da.de}

\maketitle

\begin{abstract}
A $t\da(n,K,\lambda;q)$ design, also called the $q$-analog of a $t$-wise balanced design, is a set $\B$ of subspaces with dimensions contained in $K$ of the $n$-dimensional vector space $\F_q^n$ over the finite field with $q$ elements such that each $t$-subspace of $\F_q^n$ is contained in exactly $\lambda$ elements of $\B$. In this paper we give a construction of an infinite series of nontrivial $t\da(n,K,\lambda;q)$ designs with $|K|=2$ for all dimensions $t\ge 1$ and all prime powers $q$ admitting the standard Borel subgroup as group of automorphisms. Furthermore, replacing $q=1$ gives an ordinary $t$-wise balanced design defined on sets.
\end{abstract}

\section{Introduction}

In the following let $\subspaces{\F_q^n}{k}$ denote the set of $k$-subspaces of the $n$-dimensional vector space $\F_q^n$ over the finite field $\F_q$ with $q$ elements. The expression $\qbinom{n}{k}{q}=\prod_{i=0}^{k-1}\frac{q^n-q^i}{q^k-q^i}$ counts the number of $k$-subspaces of $\F_q^n$ and it is called the $q$-\emph{binomial coefficient}. 

A \emph{$t\da(n,K,\lambda;q)$ design}, also called the \emph{$q$-analog of a $t$-wise balanced design} or $t$-wise balanced design over $\F_q$, is a set $\B$ of subspaces of $\F_q^n$ with dimensions contained in $K$ such that each $t$-subspace of $\F_q^n$ is contained in exactly $\lambda$ members of the set $\B$.

The set $\B=\cup_{k\in K}\subspaces{\F_q^n}{k}$ is already a $t\da(n,K,\lambda_{\max};q)$ design, the so-called \emph{trivial} $t$-wise balanced design, where $\lambda_{\max}=\sum_{k\in K}\qbinom{n-t}{k-t}{q}$.

If $K=\{k\}$ is a one element set $\B$ is simply called the $q$-analog of a $t$-design and write $t\da(n,k,\lambda;q)$ instead of $t\da(n,\{k\},\lambda;q)$ to indicate the parameters of the design. So far only a few results have been published on $t\da(n,k,\lambda;q)$ designs. Explicit constructions \cite{BEO+13,BKL05,Ito98,MMY95,Suz90,Suz92,Tho87,Tho96} are known for $t=1,2,3$ whereas the existence of these objects for all $t\ge 1$ was recently shown \cite{FLV13:arxiv}.

In this paper we concentrate on the case $|K|=2$ and present an explicit construction of a series of nontrivial $t\da(n,K,\lambda;q)$ designs for all $t\ge 1$ and all prime powers $q$.

\section{Groups of Automorphisms}

The general linear group $\GL(\F_q^n)$ of the vector space $\F_q^n$, whose elements are represented by $n\times n$-matrices, acts on $\subspaces{\F_q^n}{k}$ by left multiplication $\alpha S := \{\alpha x\mid x\in S\}$. 

An element $\alpha\in \GL(\F_q^n)$ is called an \emph{automorphism} of a $t\da(n,K,\lambda;q)$ design $\B$ if and only if $\B=\alpha\B := \{\alpha S\mid S\in \B\}$. The set of all automorphisms of a design forms a group, called \emph{the automorphism group} of the $t$-wise balanced design. Every subgroup of the automorphism group of a $t$-wise balanced design is denoted as \emph{a group of automorphisms} of the $t$-wise balanced design.

A construction approach derives from the Kramer and Mesner construction of ordinary $t$-designs \cite{KM76}: 

A subgroup $G$ of $\GL(\F_q^n)$ induces an equivalence relation on the set of $k$-subspaces of $\F_q^n$ by defining $S\simeq_G S':\Longleftrightarrow \exists\alpha\in G: \alpha S=S'$. The corresponding equivalence class of $S$ is called the $G$-orbit on $S$  and it is denoted by $G(S):=\{\alpha S\mid \alpha \in G\}$. The stabilizer of $S$ is abbreviated by $G_S:=\{\alpha\in G\mid \alpha S = S\}$. 

Now, a $t\text{-}(n,K,\lambda;q)$ design $\B$ admits a subgroup $G$ of the general linear $\GL(\F_q^n)$ as a group of automorphisms if and only if $\B$ consists of $G$-orbits on $\cup_{k\in K}\subspaces{\F_q^n}{k}$. 

The $G$-incidence matrix $A_{t,k}^G$ is defined to be the matrix whose rows and columns are indexed by the $G$-orbits on the set of $t$- and $k$-subspaces of $\F_q^n$, respectively. The entry indexed by the orbit $G(T)$ on $\subspaces{\F_q^n}{t}$ and by the orbit $G(S)$ on $\subspaces{\F_q^n}{k}$ is defined to be the number $|\{S'\in G(S)\mid T\subseteq S'\}|$. Note that each row of $A_{t,k}^G$ adds up to the constant value $\qbinom{n-t}{k-t}{q}$. 

If $A_{t,K}^G:=|_{k\in K}A_{t,k}^G$ denotes the concatenation of all $G$-incidence matrices $A_{t,k}^G$ for all $k\in K$ we obtain the following result: 

\begin{thm}
A $t\text{-}(n,K,\lambda;q)$ design admitting $G\le \GL(\F_q^n)$ as a group of automorphisms exists if and only if there is a $0/1$-column vector $\mathbf{x}$ satisfying $A_{t,K}^G\mathbf{x}=\lambda\mathbf{1}$, where $\mathbf{1}$ denotes the all-one column vector. The vector $\mathbf{x}$ represents the corresponding selection of $G$-orbits on the set of subspaces $\cup_{k\in K}\subspaces{\F_q^n}{k}$.
\end{thm}

For instance, Tables \ref{tab:kmr} and \ref{tab:kmr2} show a list of $q$-analogs of $t$-wise balanced designs over finite fields we constructed with the Kramer-Mesner approach using a computer search. All prescribed groups of automorphisms we used are subgroups of the normalizer of a Singer cycle, which is generated by a \emph{Singer cycle} $\sigma$ of order $q^n-1$ and the \emph{Frobenius automorphism} $\phi$ of order $n$. This group and its subgroups have already been used for the successful construction of designs over finite fields (see \cite{BEO+13,BKL05}).

\begin{table}[!htbp]
\centering
\caption{$t$-$(n,K,\lambda;q)$ designs with $|K|=2$}\label{tab:kmr}
\begin{tabular}{lp{9.1cm}}
\toprule
$t\text{-}(n,K,\lambda;q)$ & $\lambda_{\max}$; group; values for $\lambda$\\
\midrule
$2\text{-}(6,\{3,4\},\lambda;2)$ & 50;  $\langle\sigma\rangle$; 5, 8, 9, 11, 12, 14, 15, 17, 18, 20, 21, 23, 24\\
\midrule
$2\text{-}(7,\{3,4\},\lambda;2)$ & 186;  $\langle\sigma\rangle$; 5, 6, 7, 8, 9, 10, 11, 12, 13, 14, 15, 16, 17, 18, 19, 20, 21, 22, 23, 24, 25, 26, 27, 28, 29, 30, 31, 32, 33, 34, 35, 36, 37, 38, 39, 40, 41, 42, 43, 44, 45, 45, 46, 47, 48, 49, 50, 51, 52, 53, 54, 55, 56, 57, 58, 59, 60, 61, 62, 63, 64, 65, 66, 67, 68, 69, 70, 71, 72, 73, 74, 75, 76, 77, 78, 79, 80, 81, 82, 83, 84, 85, 86, 87, 88, 89, 90, 91, 92, 93, 94 \\
\midrule
$2\text{-}(8,\{3,4\},\lambda;2)$ & 714;  $\langle\sigma,\phi\rangle$; 7, 21, 28, 35, 42, 47, 56, 63, 70, 77, 84, 91, 98, 105, 112, 119, 126, 133, 140, 147, 154, 161, 168, 175, 182, 189, 196, 203, 210, 217, 224, 231, 238, 245, 252, 259, 266, 273, 280, 287, 294, 301, 308, 315, 322, 329, 336, 343, 350, 357 \\
\midrule
$2\text{-}(9,\{3,4\},\lambda;2)$ & 2794; $\langle\sigma,\phi\rangle$; 21, 22, 42, 43, 63, 64, 84, 85, 105, 106, 126, 127, 147, 148, 168, 169, 189, 190, 210, 211, 231, 232, 252, 253, 273, 274, 294, 295, 315, 316, 336, 337, 357, 358, 378, 379, 399, 400, 420, 421, 441, 442, 462, 463, 483, 484, 504, 505, 525, 526, 546, 547, 567, 568, 588, 589, 609, 610, 630, 631, 651, 652, 672, 673, 693, 694, 714, 715, 735, 736, 756, 757, 777, 778, 798, 799, 819, 820, 840, 841, 861, 862, 882, 883, 903, 904, 924, 925, 945, 946, 966, 967, 987, 988, 1008, 1009, 1029, 1030, 1050, 1051, 1071, 1072, 1092, 1093, 1113, 1114, 1134, 1135, 1155, 1156, 1176, 1177, 1197, 1198, 1218, 1219, 1239, 1240, 1260, 1261, 1281, 1282, 1302, 1303, 1323, 1324, 1344, 1345, 1365, 1366, 1387\\
\bottomrule
\end{tabular}
\end{table}

\begin{table}[!htbp]
\centering
\caption{$t$-$(n,K,\lambda;q)$ designs with $|K|=3$}\label{tab:kmr2}
\begin{tabular}{lp{8.7cm}}
\toprule
$t\text{-}(n,K,\lambda;q)$ & $\lambda_{\max}$; group; values for $\lambda$\\
\midrule
$2\text{-}(6,\{3,4,5\},\lambda;2)$ & 65; $\langle\sigma\rangle$; 23, 30\\
\midrule
$2\text{-}(7,\{3,4,5\},\lambda;2)$ & 341; $\langle\sigma,\phi\rangle$; 71, 78, 82, 85, 86, 89, 92, 93, 96, 99, 103, 106, 107, 113, 115, 119, 120, 122, 124, 126, 127, 128, 129, 130, 131, 133, 134, 135, 136, 137, 138, 139, 140, 141, 142, 143, 144, 145, 146, 147, 148, 149, 150, 151, 152, 153, 154, 155, 156, 157, 158, 159, 160, 161, 162, 163, 164, 165, 166, 167, 168, 169, 170 \\
\midrule
$3\text{-}(8,\{4,5,6\},\lambda;2)$ & 341; $\langle\sigma,\phi\rangle$; 156, 166\\
\bottomrule
\end{tabular}
\end{table}

Note, that we were also able to find a \emph{large set} of $t\da(n,K,\lambda;q)$ designs which is a set of disjoint $t\da(n,K,\lambda;q)$ designs such that their union cover the whole set $\cup_{k\in K}\subspaces{\F_q^n}{k}$ for some set of parameters:
\begin{itemize}
\item three disjoint $2\da(7,\{3,4\},62;2)$ designs
\item two disjoint $2\da(7,\{3,4\},93;2)$ designs
\item two disjoint $2\da(8,\{3,4\},357;2)$ designs
\end{itemize}

\section{Echelon Equivalence}

In the following we consider subspaces as column spaces. Let $S$ be a $k$-subspace of $\F_q^n$. A matrix $\Gamma$ having $n$ rows, $k$ columns, and entries in $\F_q$ is called a \emph{generator matrix} of $S$ if and only if the columns of $\Gamma$ yield a base of $S$. There are several generator matrices for the same $k$-subspace $S$ but using elementary Gaussian transformations of the columns yields a uniquely determined generator matrix, the \emph{canonical} generator matrix, $\Gamma_C(S)$ of the subspace $S$, having the structure shown in Table \ref{tab:canonic}.

\begin{table}[!htbp]
\centering
\caption{The structure of an Echelon matrix}\label{tab:canonic}
\[
\begin{array}{|cccccc|l}
\cline{1-1}\cline{6-6}
* & * & * & \cdots & * & * \\
\vdots & \vdots & \vdots & & \vdots & \vdots \\
* & * & * & \cdots & * & * \\
\cline{1-6}
1 & 0 & 0 & \cdots & 0 & 0 & \mbox{base row} \\
\cline{1-6}
& \multicolumn{1}{|c}{*} & * & \cdots & * & * \\
& \multicolumn{1}{|c}{\vdots} & \vdots & & \vdots & \vdots \\
& \multicolumn{1}{|c}{*} & * & \cdots & * & * \\
\cline{1-6}
& \multicolumn{1}{|c}{1} & 0 & \cdots & 0 & 0& \mbox{base row} \\
\cline{1-6}
& & \multicolumn{1}{|c}{*} & \cdots & * & * \\
& & & \ddots & & \\
& & & & \multicolumn{1}{|c}{*} & * \\
& & & & \multicolumn{1}{|c}{\vdots} & \vdots \\
& & & & \multicolumn{1}{|c}{*} & * \\
\cline{1-6}
& & & & \multicolumn{1}{|c}{1} & 0 & \mbox{base row}\\
\cline{1-6}
& & & & & \multicolumn{1}{|c|}{*} \\
& & & & & \multicolumn{1}{|c|}{\vdots} \\
& & & & & \multicolumn{1}{|c|}{*} \\
\cline{1-6}
& & & & & \multicolumn{1}{|c|}{1} & \mbox{base row}\\
\cline{1-6}
& & & & & 0 \\
& & & & & \vdots \\
& & & & & 0 \\
\cline{1-1}\cline{6-6}
\end{array}
\]
\end{table}

Hereby the stars in this matrix represent elements in $\F_q$ and the row numbers where new steps commence are called \emph{base rows indices}.

Now, we introduce an equivalence relation on the set $\subspaces{\F_q^n}{k}$ which we call \emph{Echelon equivalence}: 

Two $k$-subspaces $S$ and $S^\prime$ are defined to be \emph{Echelon equivalent}, abbreviated by $S\simeq_E S^\prime$, if and only if the base row indices of the canonical generator matrices $\Gamma_C(S)$ and $\Gamma_C(S^\prime)$ are the same.

For instance, the two $3$-subspaces of $\F_5^6$ generated by the following generator matrices are Echelon equivalent:
\[
\begin{bmatrix}
2&1 &1 \\
\bf 1&\bf0 &\bf0 \\
0&1 &4 \\
\bf0&\bf1 &\bf0\\
0&0 &2 \\
\bf0&\bf0 &\bf1 \\
\end{bmatrix}
\quad\mbox{ and }\quad
\begin{bmatrix}
4&0 &3 \\
\bf 1&\bf0 &\bf0 \\
0&0 &2 \\
\bf0&\bf1 &\bf0\\
0&0 &1 \\
\bf0&\bf0 &\bf1 \\
\end{bmatrix}
\]
Table \ref{tab:canonic2} shows all Echelon equivalence classes of $\F_5^6$.
\begin{table}[!htbp]
\centering
\caption{The Echelon forms of $\F_5^6$}\label{tab:canonic2}
$\begin{bmatrix}
1&0&0\\
0&1&0\\
0&0&1\\
0&0&0\\
0&0&0\\
0&0&0\\
\end{bmatrix}$
$\begin{bmatrix}
1&0&0\\
0&1&0\\
0&0&*\\
0&0&1\\
0&0&0\\
0&0&0\\
\end{bmatrix}$
$\begin{bmatrix}
1&0&0\\
0&1&0\\
0&0&*\\
0&0&*\\
0&0&1\\
0&0&0\\
\end{bmatrix}$
$\begin{bmatrix}
1&0&0\\
0&1&0\\
0&0&*\\
0&0&*\\
0&0&*\\
0&0&1\\
\end{bmatrix}$
$\begin{bmatrix}
1&0&0\\
0&*&*\\
0&1&0\\
0&0&1\\
0&0&0\\
0&0&0\\
\end{bmatrix}$\\
$\begin{bmatrix}
1&0&0\\
0&*&*\\
0&1&0\\
0&0&*\\
0&0&1\\
0&0&0\\
\end{bmatrix}$
$\begin{bmatrix}
1&0&0\\
0&*&*\\
0&1&0\\
0&0&*\\
0&0&*\\
0&0&1\\
\end{bmatrix}$
$\begin{bmatrix}
1&0&0\\
0&*&*\\
0&*&*\\
0&1&0\\
0&0&1\\
0&0&0\\
\end{bmatrix}$
$\begin{bmatrix}
1&0&0\\
0&*&*\\
0&*&*\\
0&1&0\\
0&0&*\\
0&0&1\\
\end{bmatrix}$
$\begin{bmatrix}
1&0&0\\
0&*&*\\
0&*&*\\
0&*&*\\
0&1&0\\
0&0&1\\
\end{bmatrix}$\\
$\begin{bmatrix}
*&*&*\\
1&0&0\\
0&1&0\\
0&0&1\\
0&0&0\\
0&0&0\\
\end{bmatrix}$
$\begin{bmatrix}
*&*&*\\
1&0&0\\
0&1&0\\
0&0&*\\
0&0&1\\
0&0&0\\
\end{bmatrix}$
$\begin{bmatrix}
*&*&*\\
1&0&0\\
0&1&0\\
0&0&*\\
0&0&*\\
0&0&1\\
\end{bmatrix}$
$\begin{bmatrix}
*&*&*\\
1&0&0\\
0&*&*\\
0&1&0\\
0&0&1\\
0&0&0\\
\end{bmatrix}$
$\begin{bmatrix}
*&*&*\\
1&0&0\\
0&*&*\\
0&1&0\\
0&0&*\\
0&0&1\\
\end{bmatrix}$\\
$\begin{bmatrix}
*&*&*\\
1&0&0\\
0&*&*\\
0&*&*\\
0&1&0\\
0&0&1\\
\end{bmatrix}$
$\begin{bmatrix}
*&*&*\\
*&*&*\\
1&0&0\\
0&1&0\\
0&0&1\\
0&0&0\\
\end{bmatrix}$
$\begin{bmatrix}
*&*&*\\
*&*&*\\
1&0&0\\
0&1&0\\
0&0&*\\
0&0&1\\
\end{bmatrix}$
$\begin{bmatrix}
*&*&*\\
*&*&*\\
1&0&0\\
0&*&*\\
0&1&0\\
0&0&1\\
\end{bmatrix}$
$\begin{bmatrix}
*&*&*\\
*&*&*\\
*&*&*\\
1&0&0\\
0&1&0\\
0&0&1\\
\end{bmatrix}$
\end{table}

If $i$ is the number of stars in the matrix then $q^i$ is the number of $k$-subspaces in the corresponding equivalence class. The maximum number of stars is $(n-k)k$. 

Replacing the stars by arbitrary elements of $\F_q$ we get a canonical transversal of the Echelon equivalence classes. In order to determine the number of different Echelon equivalence classes we have to calculate the number of possibilities to choose $k$ base row indices in a matrix with $n$ rows, which is $\binom{n}{k}$. Hence, each Echelon equivalence class corresponds to a unique $k$-subset of the set of possible row indices $\{1,\ldots, n\}$. 

More formally, let $e_j$ denote the unit vector having exactly one entry $1$ in the $j$th position and $0$ in the remaining positions. Representatives of the Echelon equivalence classes are given by generator matrices consisting of unit vectors $\Gamma(\pi):=[e_{\pi_1}|\ldots|e_{\pi_k}]$ where $\pi=\{\pi_1,\ldots,\pi_k\}$ is a $k$-subset of $\{1,\ldots,n\}$. The corresponding subspace is denoted by $E(\pi):=\langle e_{\pi_1},\ldots,e_{\pi_k}\rangle$.

For example, the Echelon equivalence class containing
\[
\Gamma(\pi)=
[e_2|e_3|e_6]=
\begin{bmatrix}
0&0&0 \\
1&0&0 \\
0&0&0 \\
0&1&0\\
0&0&0 \\
0&0&1 \\
\end{bmatrix}
\]
corresponds to the subset $\pi=\{2,3,6\}$. 

\section{Parabolic and Borel Subgroups}

A \emph{flag} $[S_1,\ldots,S_r]$ of $\F_q^n$ of length $r$ is a sequence of subspaces of $\F_q^n$ satisfying $\{0\}\subsetneq S_1\subsetneq\ldots \subsetneq S_r\subseteq \F_q^n$. The maximum length $r$ of a flag is $n$. In this case the sequence of subspaces satisfies $\dim(S_i)=i$ and $S_n=\F_q^n$. Flags of length $n$ are called \emph{complete}. The complete flag $[E_1,\ldots,E_n]$ consisting of all standard subspaces $E_i:=\langle e_1,\ldots,e_i\rangle$ is called the \emph{standard} flag. 

We can also define a group action of $\GL(\F_q^n)$ on the set of flags of $\F_q^n$ by setting $\alpha[S_1,\ldots,S_r]:=[\alpha S_1,\ldots,\alpha S_r]$ for all $\alpha\in\GL(\F_q^n)$ and all flags $[S_1,\ldots,S_r]$.

The stabilizer of a flag $[S_1,\ldots,S_r]$ is also called \emph{parabolic} subgroup \cite{AB95} and it is the intersection of all single stabilizers of $S_i$:
\[
\GL(\F_q^n)_{[S_1,\ldots,S_r]}=\bigcap_{i=1}^r\GL(\F_q^n)_{S_i}
\]

In order to determine the stabilizer of the standard flag we first have to consider the stabilizer of the standard subspace $E_i$ which consists of all matrices of the form
\[
\left[\begin{array}{c|c}
A & B\\
\hline
\mathbf 0 & C
\end{array}
\right]
\]
where $A\in\GL(\F_q^i)$, $B\in\F_q^{i\times n-i}$ and $C\in\GL(\F_q^{n-i})$. The intersection of the stabilizers of all $E_i$ yields the stabilizer of the standard flag which is the set of all nonsingular upper triangular matrices---the so called \emph{standard Borel subgroup} \cite{Bor56}:
\[
\borel(\F_q^n):=\GL(\F_q^n)_{[E_1,\ldots,E_n]}=\bigcap_{i=1}^n\GL(\F_q^n)_{E_i}
\]

Now we establish the following connection between Echelon equivalence and the standard Borel subgroup \cite{Bra09}: 

The multiplication of  $\alpha\in \borel(\F_q^n)$ to a $k\times n$ matrix $\Gamma$ from the left, is equivalent to a series of elementary Gaussian transformations of the rows of $\Gamma$ such that rows will be multiplied by a nonzero finite field element or such that a multiple of a first row will be added to a second row above to the first row. Applying an arbitrary element $\alpha\in \borel(\F_q^n)$ to a canonic generator matrix $\Gamma_C(S)$ from the left, $\alpha\Gamma_C(S)  = \Gamma,$ yields a generator matrix $\Gamma$ of a subspace $S'$ whose base row indices are the same as in $\Gamma_C(S)$. Moreover, the Echelon equivalence and the equivalence induced by the group action of the Borel subgroup are the same:
\[
S\simeq_E S'\iff S\simeq_{\borel(\F_q^n)} S'
\]

This also yields that representatives of the $\borel(\F_q^n)$-orbits on the set of $k$-subspaces of $\F_q^n$ can be obtained from $k$-subsets of $\{1,\ldots,n\}$:

\begin{cor}
Representatives of the $\borel(\F_q^n)$-orbits on the set $\subspaces{\F_q^n}{k}$ are given by the subspaces $E(\pi)$ and their generator matrices $\Gamma(\pi)$ where $\pi$ is a $k$-subset of $\{1,\ldots,n\}$.
\end{cor}

In the following we consider some properties of the $\borel(\F_q^n)$-incidence matrices. The following property is immediate from the Echelon equivalence classes:

\begin{lem}\label{lem:subset}
The entry $a_{\tau,\pi}$ of the incidence matrix $A_{t,k}^{\borel(\F_q^n)}$ whose row is indexed by the $\borel(\F_q^n)$-orbit on $E(\tau)$ and whose column is indexed by the $\borel(\F_q^n)$-orbit on $E(\pi)$ is nonzero if and only if $\tau\subseteq \pi$. In this case, if $a_{\tau,\pi}$ is nonzero it is a power of $q$.
\end{lem}

\begin{cor}
If we substitute $q=1$ in $A_{t,k}^{\borel(\F_q^n)}$ we obtain the incidence matrix between all $t$-subsets and $k$-subsets of $\{1,\ldots,n\}$.
\end{cor}

Considering the orbits of the standard Borel subgroup on subspaces or the Echelon equivalence classes, respectively, might be understood as the proper form of \lq\lq $q$-analogization\rq\rq: If the number of stars in an Echelon equivalence class representative---as depicted in Table \ref{tab:canonic2}---is $i$ each star can be substituted by a finite field element which yields the cardinality $q^i$ of this particular class. Setting $q=1$ means, that we replace all stars by $0$'s and each Echelon class contains only one element. Hence, the Echelon equivalence classes themselves can be considered as subsets. Furthermore, the incidence matrix $A_{t,k}^{\borel(\F_q^n)}$ becomes the incidence matrix between subsets for $q=1$.

\begin{lem}\label{lem:blockform}
The matrix $A_{t,\{t+1,t+2\}}^{\borel(\F_q^n)}=A_{t,t+1}^{\borel(\F_q^n)}\mid A_{t,t+2}^{\borel(\F_q^n)}$ has the following form:
\[
A_{t,\{t+1,t+2\}}^{\borel(\F_q^n)}=
\left[
\begin{array}{ccc|ccc|ccc}
q^{n-t-1}\hspace{-1cm}& & & & & & * & \cdots & *\\
& \ddots & & & A_{t,t+2}^{\borel(\F_q^{n-1})}\hspace{-0.5cm} & &\vdots & & \vdots \\
& &\hspace{-0.5cm}q^{n-t-1} & & & & * & \cdots & *\\ 
\hline
& & & 0 &\cdots & 0 & * & \cdots& *\\
& A_{t-1,t}^{\borel(\F_q^{n-1})}\hspace{-0.5cm} & & \vdots &  & \vdots & \vdots & & \vdots\\
& & & 0 & \cdots & 0 & * & \cdots &* \\
\end{array}\right]
\]
The first block of columns is indexed by all $(t+1)$-subsets of $\{1,\ldots,n\}$ containing $6$, the second block of columns is indexed by all $(t+2)$-subsets of $\{1,\ldots,n-1\}$, all remaining $(t+1)$- and $(t+2)$-subsets of $\{1,\ldots,n\}$ occur in the third block of columns. The first block of rows is indexed by all $t$-subsets of  $\{1,\ldots,n-1\}$  and the second block of rows is indexed by all $t$-subsets of $\{1,\ldots,n\}$ containing $6$.
\end{lem}

\begin{proof}
The proof is straightforward and uses Lemma \ref{lem:subset}. We only look at the block in the upper left corner. The rows correspond to the $t$-subsets of $\{1,\ldots,n-1\}$ and the columns correspond to $(t+1)$-subsets of $\{1,\ldots,n\}$ containing the element $6$ which arise by all $t$-subsets of $\{1,\ldots,n-1\}$ by just adding the element $6$. In order to determine the matrix entry whose row is indexed by the $t$-subset $\{\pi_1,\ldots,\pi_t\}$ and whose column is indexed by the $(t+1)$-subset $\{\pi_1,\ldots,\pi_t,6\}$ we have to count the number of canonical generator matrices of the form $[e_{\pi_1}|\ldots|e_{\pi_t}|v]$ where $v$ must have the entry $1$ in the last position. Since the last entry of $v$ is $1$ and since $t$ entries are $0$ due to the remaining $t$ base row indices we have $n-t+1$ positions in $v$ for which we can choose any finite field element. Finally, we get $q^{n-t-1}$ as corresponding incidence matrix entry.
\end{proof}

\section{Construction}

In this section we finally describe the construction of an infinite family of $$t\da(t+4,\{t+1,t+2\},q^3+q^2+q+1;q)$$ designs as a union of certain $\borel(\F_q^n)$-orbits on $(t+1)$- and $(t+2)$-subspaces of $\F_q^n$. 

As selection of $\borel(\F_q^n)$-orbits we choose the representatives $\T$ of $\borel(\F_q^n)$-orbits belonging to the first two blocks of columns of $A_{t,\{t+1,t+2\}}^{\borel(\F_q^n)}$ as they are given in Lemma \ref{lem:blockform}, i.\,e. we get the following set of subspaces:
\[
\T:=\{E(\pi\cup\{6\})\mid \pi\in \mbox{${\{1,\ldots,n-1\}\choose t+1}$}\}
\cup\{E(\pi)\mid \pi\in \mbox{${\{1,\ldots,n-1\}\choose t+2}$}\}
\]
and the union of orbits:
\[
\B:=\bigcup_{S\in\T}\borel(\F_q^n)(S)
\]
By this selection the following columns of $A_{t,\{t+1,t+2\}}^{\borel(\F_q^n)}$ are chosen:
\[M=
\left[
\begin{array}{ccc|ccc}
q^{n-t-1}\hspace{-1cm}& & & & &\\
& \ddots & & & A_{t,t+2}^{\borel(\F_q^{n-1})}\hspace{-0.5cm} &\\
& &\hspace{-0.5cm}q^{n-t-1} & & &\\ 
\hline
& & & 0 &\cdots &0\\
& A_{t-1,t}^{\borel(\F_q^{n-1})}\hspace{-0.5cm} & & \vdots &  &\vdots\\
& & & 0 & \cdots & 0\\
\end{array}\right]
\]
The final issue is now to investigate under which conditions the set $\B$ becomes a $t$-wise balanced $t\da(n,\{t+1,t+2\},\lambda;q)$ design. To answer this question we determine the row sum of the selected matrix. For the first block of rows we get the row sum: 
\[
\alpha=q^{n-t-1}+\qbinom{(n-1)-t}{(t+2)-t}{q} = q^{n-t-1}+\qbinom{n-t-1}{2}{q}
\]
and for the second block of rows the rows add up to:
\[
\beta=\qbinom{(n-1)-(t-1)}{t-(t-1)}{q}=\qbinom{n-t}{1}{q}
\]
It is clear that the corresponding selection of $\borel(\F_q^n)$-orbits and columns, respectively, becomes a $t$-wise balanced $t\da(n,\{t+1,t+2\},\lambda;q)$ design if and only if both values are equal $\alpha=\beta$. In this case we get the index: 
\[
\lambda=\alpha=\beta = \qbinom{n-t}{1}{q} = q^{n-t-1}+\ldots + q^2+q+1
\]
It is easy to see that $\alpha=\beta$ if and only if $\qbinom{n-t-1}{1}{q}=\qbinom{n-t-1}{2}{q}$. From the symmetry of the $q$-binomial coefficient we get $n-t-1 = 3$ and hence: 
\[
n = t+4
\]

This shows that $\B$ really defines a $t\da(t+4,\{t+1,t+2\},q^3+q^2+q+1;q)$ design.

Furthermore, in $M$ each rom sum is equal to $q^3+q^2+q+1$ which means that exactly four entries are nonzero. Hence, substituting $q=1$ in the matrix $M$ we obtain a constant row sum of $4$. The subsets corresponding to the columns of $M$ finally defines an ordinary nontrivial $t$-wise balanced design with parameters $$t\da(t+4,\{t+1,t+2\},4)$$ for all positive integers $t\ge 1$.

Finally, we show an example: 

We construct a $2\da(6,\{3,4\},15;2)$ design. The set of representatives $\T$ of the chosen $B(\F_2^6)$-orbits on $3$- and $4$-subspaces is the following one:
\begin{center}
$E(\{1,2,6\})$, $E(\{1,3,6\})$, $E(\{1,4,6\})$, $E(\{1,5,6\})$, $E(\{2,3,6\})$, $E(\{2,4,6\})$, $E(\{2,5,6\})$, $E(\{3,4,6\})$, $E(\{3,5,6\})$, $E(\{4,5,6\})$, $E(\{1,2,3,4\})$, $E(\{1,2,3,5\})$, $E(\{1,2,4,5\})$, $E(\{1,3,4,5\})$, $E(\{2,3,4,5\})$
\end{center}
If the list
\begin{center}
$E(\{1,2\})$, $E(\{1,3\})$, $E(\{1,4\})$, $E(\{1,5\})$, $E(\{2,3\})$, $E(\{2,4\})$, $E(\{2,5\})$, $E(\{3,4\})$, $E(\{3,5\})$, $E(\{4,5\})$, $E(\{1,6\})$, $E(\{2,6\})$, $E(\{3,6\})$, $E(\{4,6\})$, $E(\{5,6\})$,
\end{center}
denotes the representatives of $B(\F_2^6)$-orbits on the set of $2$-subspaces, we get the following incidence matrix $M$ between the orbits on $2$-subspaces and the selected orbits on $3$- and $4$-subspaces:
\[M=
\left[
\begin{array}{cccccccccc|ccccc}
8&&&&&&&&&& 1&2&4& &\\
&8&&&&&&&&& 1&2& &4&\\
&&8&&&&&&&& 1& &2&4&\\
&&&8&&&&&&&  &1&2&4&\\
&&&&8&&&&&& 1&2& & &4\\
&&&&&8&&&&& 1& &2& &4\\
&&&&&&8&&&&  &1&2& &4\\
&&&&&&&8&&& 1& & &2&4\\
&&&&&&&&8&&  &1& &2&4\\
&&&&&&&&&8&  & &1&2&4\\
\hline
1&2&4&8& & & & & & &&&&\\
1& & & &2&4&8& & & &&&&\\
 &1& & &2& & &4&8& &&&&\\
 & &1& & &2& &4& &8&&&&\\
 & & &1& & &2& &4&8&&&&
\end{array}
\right]
\]
The row sum in each row is exactly $15$ which shows that the selected set of orbit representatives on $3$- and $4$-subspaces yields a $2\da(6,\{3,4\},15;2)$ design. 

Moreover, the given set of $3$- and $4$-subsets defines a $2\da(6,\{3,4\},4)$ design.


\end{document}